\documentclass[11pt,a4paper]{article}
\usepackage{amsmath,amssymb,amsthm,amsfonts,latexsym,graphicx,subfigure}
\usepackage[all,import]{xy}
\usepackage[numbers,sort&compress]{natbib}
\usepackage{mathrsfs}
\usepackage{indentfirst}
\usepackage[inline]{enumitem}
\usepackage{fancyhdr,enumerate}
\usepackage{bbm,color}
\usepackage[inline]{enumitem}
\usepackage{tikz}
\usetikzlibrary{decorations.pathreplacing,calc}
\usetikzlibrary{shapes.geometric, arrows}
\usepackage{accents,cases}
\usepackage[T1]{fontenc}
\usepackage{lmodern}
\usepackage{multirow}
\usepackage[lined, ruled, linesnumbered, longend]{algorithm2e}
\usepackage[labelsep=space]{caption}
\usepackage{array,float}
\usepackage{hyperref}
\newcommand{\PreserveBackslash}[1]{\let\temp=\\#1\let\\=\temp}
\newcolumntype{C}[1]{>{\PreserveBackslash\centering}p{#1}}
\newcolumntype{R}[1]{>{\PreserveBackslash\raggedleft}p{#1}}
\newcolumntype{L}[1]{>{\PreserveBackslash\raggedright}p{#1}}

\makeatletter
\def\wbar{\accentset{{\cc@style\underline{\mskip8mu}}}}

\makeatother

\theoremstyle{plain}
\newtheorem{theorem}{Theorem}

\newtheorem{Conj}{Conjecture}[section]
\newtheorem{lemma}{Lemma}
\newtheorem{remark}{Remark}[section]

\newtheorem{Question}{Question}

\tikzset{ vtx/.style={ circle, fill=black, inner sep=1.7pt } }

\allowdisplaybreaks

\begin{document}

\title{New bounds for equiangular lines and Balla's conjecture}
\author{Chuanyuan Ge\textsuperscript{a,\ensuremath{\ast}} \and Shiping Liu\textsuperscript{b} }

\footnotetext[0]{ \textsuperscript{*}Corresponding author.}
\footnotetext[0]{\textsuperscript{a}School of Mathematics and Statistics, Fuzhou University, Fuzhou 350108, China.\\
Email addresses:
{\tt cyge@fzu.edu.cn} }
\footnotetext[0]{\textsuperscript{b}School of Mathematical Sciences,
University of Science and Technology of China, Hefei 230026, China. \\
Email addresses:
{\tt spliu@ustc.edu.cn}
}
\date{}\maketitle

\begin{abstract}
Let $N_{\alpha}(d)$ denote the maximum number of equiangular lines in $\mathbb{R}^d$ with common angle $\arccos(\alpha)$. Balla conjectured that, if the spectral radius order $\kappa_{\frac{1-\alpha}{2\alpha}}$ of $\frac{1-\alpha}{2\alpha}$ is finite, then 
  $$N_{\alpha}(d)\leq \max\left\{\frac{(1-\alpha^2)(1-2\alpha^2)}{2\alpha^4},\left\lfloor\frac{\kappa_{\frac{1-\alpha}{2\alpha}}(d-1)}{\kappa_{\frac{1-\alpha}{2\alpha}}-1}\right\rfloor\right\},$$
 for  any $d\geq 1$. The conjecture has previously been verified only for $\alpha\in\left\{\frac{1}{3},\frac{1}{5},\frac{1}{1+2\sqrt{2}}\right\}$. In this paper, we prove that this conjecture holds for $\alpha=\frac{1}{1+2\sqrt{3}}$ and $\alpha=\sqrt{5}-2$. On the other hand, we show that Balla's conjecture  fails for infinitely many $\alpha$.
\end{abstract}

\section{Introduction}
A set of lines in $\mathbb{R}^d$ passing through the origin is called equiangular if the angle between any two lines is the same. Equiangular lines and their variants are closely related to a variety of research topics, see \cite{Balla-Draxler-Keevash-Sudakov-18,balla2024equiangular-exponential-regime,balla2021equiangular,Cao-Koolen-Lin-Yu,Jiang-Tidor-Yao-Zhang-Zhao-21,Glazyrin-Yu-18} and the references therein.


Let $N_{\alpha}(d)$ denote the maximum cardinality of a set of equiangular lines in $\mathbb{R}^d$ with common angle $\arccos(\alpha)$. In 1973, Lemmens and Seidel \cite{Lemmens-Seidel-73} initiated a study of quantity $N_{\alpha}(d)$. They completely determined the values of $N_{1/3}(d)$ for all $d\geq 1$, and conjectured that 
\[N_{\frac{1}{5}}(d)=\begin{cases}
    276 & \ 23\leq d\leq 185,\\
    \left\lfloor \frac{3(d-1)}{2}\right\rfloor & \ d\geq 185.
\end{cases}\]
 Neumaier \cite{Neumaier-89} proved that this conjecture holds for sufficiently large $d$. In 2022, Cao, Koolen, Lin, and Yu \cite{Cao-Koolen-Lin-Yu} completely confirmed this conjecture.
The authors \cite{ge-liu-2026lemmens} have found a new proof of this conjecture using nodal domain estimates .

 For general $\alpha$, Bukh \cite{Bukh-16} proved that $N_{\alpha}(d)\leq C_{\alpha}d$ for $d\geq 1$, where $C_{\alpha}$ is a constant that depends only on $\alpha$. Balla, Dräxler, Keevash, and Sudakov \cite{Balla-Draxler-Keevash-Sudakov-18} proved that $N_{\alpha}(d)\leq 1.93d$ for $d$ sufficiently large and $\alpha\neq 1/3$. In 2020, Jiang and Polyanskii \cite{Jiang-Polyanskii-20} introduced the notion of \emph{spectral radius order} and used it to determine the limit
\begin{equation}\label{eq:limit}
\lim_{d\to\infty}\frac{N_{\alpha}(d)}{d},   
\end{equation}
for \(\ \alpha\in\left(\frac{1}{1+2\sqrt{2+\sqrt{5}}},1\right)\).
 The spectral radius order $\kappa_\lambda$ of a real number $\lambda\in (0,\infty)$, is defined to be the smallest integer $k$, such that there exists a graph $G$ with $k$ vertices and largest adjacency eigenvalue $\lambda$.
 If no graph has largest adjacency eigenvalue $\lambda$, we set $\kappa_{\lambda}=\infty$. The following theorem was established by
Jiang, Tidor, Yao, Zhang, and Zhao \cite{Jiang-Tidor-Yao-Zhang-Zhao-21},
and yields a complete determination of the limit
\eqref{eq:limit}
for all \(\alpha\in(0,1)\).
\begin{theorem}[\cite{Jiang-Tidor-Yao-Zhang-Zhao-21}]\label{thm:JTYZZ21}
Let \(\alpha\in(0,1)\) and set \(\lambda:=\frac{1-\alpha}{2\alpha}\).
If \(\kappa_{\lambda}<\infty\), then
    \begin{equation}\label{eq:kappa_lambda_finite}
    N_{\alpha}(d)=\left\lfloor\frac{\kappa_{\lambda}(d-1)}{\kappa_{\lambda}-1}\right\rfloor,   
    \end{equation}
for \(d>2^{2^{C\kappa_\lambda/\alpha}}\),  where \(C\) is an absolute constant.
If, otherwise, \(\kappa_{\lambda}=\infty\), then
    \(
    N_{\alpha}(d)=d+o_{\alpha}(d).
    \)
\end{theorem}
We remark that the absolute constant \(C\) in Theorem~\ref{thm:JTYZZ21} is quite large, and the lower bound on \(d\) required for \eqref{eq:kappa_lambda_finite} is far from optimal. For example, when \(\alpha=1/3\), one has \(\kappa_\lambda=\kappa_1=2\), and \eqref{eq:kappa_lambda_finite} already holds for \(d\geq 28\) by the result of Lemmens and Seidel \cite{Lemmens-Seidel-73}. When \(\alpha=1/5\), one has \(\kappa_\lambda=\kappa_2=3\), and \eqref{eq:kappa_lambda_finite} holds for \(d\geq 185\) by the result of Cao, Koolen, Lin, and Yu \cite{Cao-Koolen-Lin-Yu}.

Thus, it is natural to ask whether the condition
\(
  d>2^{2^{C\kappa_\lambda/\alpha}}
\)
can be substantially improved. Balla \cite{balla2021equiangular} improved this condition to
\(
  d>2^{(1/\alpha)^{C\kappa_\lambda}}.
\)
When \(\alpha=\frac{1}{2k+1}\) for some \(k\in\mathbb Z_{>0}\), Balla and Buci\'c \cite{balla2024equiangular-exponential-regime} further improved it to
\(
  d>2^{(1/\alpha)^{20}}.
\) If $\alpha\geq \alpha_0>\frac{1}{1+3\sqrt{2}}$, the authors \cite{ge-liu-2024equiangular} improved it to $d>C_{\alpha_0}\kappa_{\lambda}$, where $C_{\alpha_0}$ is a constant depending only on $\alpha_0$.
Balla \cite{balla2021equiangular} also proposed the following conjecture.
\begin{Conj}[{\cite[Conjecture 1.7]{balla2021equiangular}}]\label{conj:balla}
   Let $\alpha\in (0,1)$ and $\lambda:=\frac{1-\alpha}{2\alpha}$. If $\kappa_{\lambda}<\infty$, then $$N_{\alpha}(d)\leq \max\left\{\frac{(1-\alpha^2)(1-2\alpha^2)}{2\alpha^4},\left\lfloor\frac{\kappa_{\lambda}(d-1)}{\kappa_{\lambda}-1}\right\rfloor\right\},$$
for any $d\geq 1$.
\end{Conj}
The conjecture is known to hold for \(\alpha=1/3\) and \(\alpha=1/5\) by
\cite{Cao-Koolen-Lin-Yu,Lemmens-Seidel-73}, and for
\[
\alpha=\frac{1}{1+2\sqrt{2}}
\]
by recent work of Gossett, Jiang, Teets, and Wellner
\cite{gossett-2026}. Balla
\cite{balla2021equiangular} proved that, if \(\kappa_\lambda<\infty\), then
\[
  N_{\alpha}(d)
  \leq
  \max\left\{
    \frac{(1-\alpha^2)(1-2\alpha^2)}{2\alpha^4},
    2d
  \right\}
\]
for every \(d\geq 1\).
 
 We remark that for $\alpha\in(0,1)$, $\kappa_{\frac{1-\alpha}{2\alpha}}<4$ if and only if $\alpha\in\left\{\frac{1}{3},\frac{1}{5},\frac{1}{1+2\sqrt{2}}\right\}$. Therefore, Conjecture~\ref{conj:balla} is already known in all cases of spectral radius order less than \(4\). In this paper, we study Conjecture \ref{conj:balla} when the spectral radius order of $\frac{1-\alpha}{2\alpha}$ equals $4$. We prove that Conjecture \ref{conj:balla} holds for
\[
  \alpha=\frac{1}{1+2\sqrt{3}}
  \qquad\text{and}\qquad
  \alpha=\sqrt{5}-2.
\]
We note that
\[
  \left\lfloor
    \frac{(1-\alpha^2)(1-2\alpha^2)}{2\alpha^4}
  \right\rfloor
  =
  \begin{cases}
    169, & \alpha=\frac{1}{1+2\sqrt{3}},\\
    135, & \alpha=\sqrt{5}-2,
  \end{cases}
\]
and that \(\kappa_{\frac{1-\alpha}{2\alpha}}=4\) for both values of \(\alpha\).
Thus Balla's conjecture gives the following explicit bounds, which are the main
positive results of this paper.
\begin{theorem}\label{thm:main}
For every positive integer \(d\), we have
\[
  N_{\frac{1}{1+2\sqrt{3}}}(d)
  \leq
  \max\left\{
    169,
    \left\lfloor \frac{4(d-1)}{3}\right\rfloor
  \right\},
\]
and
\[
  N_{\sqrt{5}-2}(d)
  \leq
  \max\left\{
    135,
    \left\lfloor \frac{4(d-1)}{3}\right\rfloor
  \right\}.
\]
\end{theorem}
On the other hand, we show that Conjecture \ref{conj:balla} fails for infinitely many values of \(\alpha\). 
\begin{theorem}\label{thm:counterexample}
    Let $\alpha_n=\frac{1}{1+4\cos\frac{\pi}{n}}$ and $\lambda_n=\frac{1-\alpha_n}{2\alpha_n}$. There exists an constant $n_0$, such that for any odd integer $n>n_0$, we have $\kappa_{\lambda_n}<\infty$ and $$N_{\alpha_n}(2n-2)> \max\left\{\frac{(1-\alpha_n^2)(1-2\alpha_n^2)}{2\alpha_n^4},\left\lfloor\frac{\kappa_{\lambda_n}(2n-3)}{\kappa_{\lambda_n}-1}\right\rfloor\right\}.$$
\end{theorem}

\section{Preliminaries}
Let $G=(V,E)$ be a graph. We write $u\sim v$ if $\{u,v\}\in E$, denote by $\deg_G(u)$ the degree of $u$ in $G$, and write $\Delta_G$ for the maximum degree of $G$. For a subset $V_0\subseteq V$, let
\[
N_G(V_0):=\{u\in V:\text{ there exists }v\in V_0\text{ with }u\sim v\},
\]
and set $\overline N_G(V_0):=V_0\cup N_G(V_0)$. We denote by $G[V_0]$ the subgraph of $G$ induced by $V_0$.

We denote by $P_n$ the path graph with $n$ vertices and by $C_n$ the cycle graph with $n$ vertices. For \(n\geq 4\), let \(D_n\) be the tree obtained from the path
\(v_1v_2\cdots v_{n-1}\) by attaching a pendant vertex \(u\) to \(v_2\), and
let \(E_n\) be the tree obtained from the same path by attaching a pendant
vertex \(u\) to \(v_3\). The graphs \(D_n\) and \(E_n\) are illustrated in
Figure~\ref{fig:DnEn}.
\begin{figure}[htbp]
    \centering

    \begin{minipage}{0.45\textwidth}
        \centering
        $D_n$

        \vspace{0.3cm}

        \begin{tikzpicture}[scale=1]
            \tikzset{
                vtx/.style={
                    circle,
                    fill=black,
                    inner sep=1.6pt
                }
            }

            \node[vtx, label=below:{$v_1$}] (a) at (0,0) {};
            \node[vtx, label=below:{$v_2$}] (b) at (1,0) {};
            \node[vtx, label=above:{$u$}]   (c) at (1,1) {};
            \node[vtx, label=below:{$v_3$}] (d) at (2,0) {};
            \node[vtx, label=below:{$v_{n-2}$}] (e) at (4,0) {};
            \node[vtx, label=below:{$v_{n-1}$}] (f) at (5,0) {};

            \draw (a)--(b);
            \draw (b)--(c);
            \draw (b)--(d);
            \draw[dashed] (d)--(e);
            \draw (e)--(f);
        \end{tikzpicture}
    \end{minipage}
    \hfill
    \begin{minipage}{0.45\textwidth}
        \centering
        $E_n$

        \vspace{0.3cm}

        \begin{tikzpicture}[scale=1]
            \tikzset{
                vtx/.style={
                    circle,
                    fill=black,
                    inner sep=1.6pt
                }
            }

            \node[vtx, label=below:{$v_1$}] (a) at (0,0) {};
            \node[vtx, label=below:{$v_2$}] (b) at (1,0) {};
            \node[vtx, label=below:{$v_3$}] (c) at (2,0) {};
            \node[vtx, label=above:{$u$}]   (u) at (2,1) {};
            \node[vtx, label=below:{$v_4$}] (d) at (3,0) {};
            \node[vtx, label=below:{$v_{n-2}$}] (e) at (5,0) {};
            \node[vtx, label=below:{$v_{n-1}$}] (f) at (6,0) {};

            \draw (a)--(b)--(c)--(d);
            \draw[dashed] (d)--(e);
            \draw (e)--(f);
            \draw (c)--(u);
        \end{tikzpicture}
    \end{minipage}

    \caption{The graphs $D_n$ and $E_n$.}
    \label{fig:DnEn}
\end{figure}
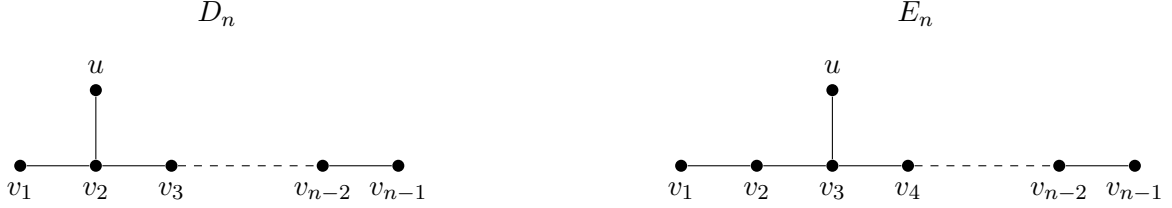

We denote by \(A_G\) the adjacency matrix of a graph \(G=(V,E)\). Its
eigenvalues are ordered as
\[
  \lambda_1(A_G)\geq \lambda_2(A_G)\geq \cdots \geq \lambda_n(A_G),
\]
where \(n=|V|\). We write \(m_\lambda(A_G)\) for the multiplicity of
\(\lambda\) as an eigenvalue of \(A_G\). As usual, \(I_n\) and \(J_n\) denote
the identity matrix and the all-one matrix of order \(n\), respectively, and
\(\mathbf 1_n\) denotes the all-one vector of dimension \(n\). When no confusion
arises, we omit the subscript \(n\).

We recall the following fundamental results from spectral graph theory, see, for example, \cite[Sections 1.4.3, 1.4.4, 3.1.1]{Brouwe-Haemers-12}.

\begin{lemma}\label{thm:smith}
Let $G$ be a connected graph. Then
$\lambda_1(A_G)<2$
if and only if $G$ is isomorphic to one of the following graphs:
$$
P_n\quad (n\geq 1),\qquad
D_n\quad (n\geq 4),\qquad
E_6,\ E_7,\ E_8.
$$
Furthermore, we have
\[
\lambda_1(P_n)=2\cos\frac{\pi}{n+1},
\qquad
\lambda_1(D_n)=2\cos\frac{\pi}{2(n-1)},
\]
\[
\lambda_1(E_6)=2\cos\frac{\pi}{12},\quad
\lambda_1(E_7)=2\cos\frac{\pi}{18},\quad
\lambda_1(E_8)=2\cos\frac{\pi}{30}.
\]
\end{lemma}
\begin{lemma}\label{lemma:Cn}
Let \(C_n\) be the cycle graph on \(n\geq 3\) vertices. Then the eigenvalues
of its adjacency matrix \(A_{C_n}\) are
\[
  2\cos\left(\frac{2\pi k}{n}\right),
  \qquad k=0,1,\ldots,n-1,
\]
counted with multiplicity. 
\end{lemma}

A spherical \(\{-\alpha,\alpha\}\)-code in \(\mathbb R^d\) is a set of unit
vectors
\[
  \mathscr C=\{v_1,v_2,\dots,v_n\}\subseteq \mathbb R^d
\]
such that \(\langle v_i,v_j\rangle\in\{-\alpha,\alpha\}\) whenever \(i\neq j\).
We denote its Gram matrix by
\[
  M_{\mathscr C}:=\bigl(\langle v_i,v_j\rangle\bigr)_{1\leq i,j\leq n}.
\]
To a spherical code, we associate a graph \(G=(V,E)\) by setting
\begin{equation}\label{eq:graphs_associated_code}
  V=\{v_1,v_2,\dots,v_n\},
  \qquad
  E=\bigl\{\{v_i,v_j\}:\langle v_i,v_j\rangle=-\alpha\bigr\}.
\end{equation}
A direct computation gives
\begin{equation}\label{eq:M_C}
  \frac{1}{2\alpha}M_{\mathscr C}
  =
  \frac{1-\alpha}{2\alpha}I+\frac{1}{2}J-A_G.
\end{equation}
The following lemma relates equiangular lines, spherical codes, and the associated graphs.
\begin{lemma}[see, e.g., {\cite[Lemma 1]{ge-liu-2026lemmens}}]\label{lemma:equivalent}
For fixed $\alpha\in(0,1)$, the following are equivalent.
\begin{itemize}
    \item[$(i)$] There exist $n$ equiangular lines in $\mathbb{R}^d$ with common angle $\arccos(\alpha)$.
    \item[$(ii)$] There exists a spherical $\{-\alpha,\alpha\}$-code of size $n$ in $\mathbb{R}^d$.
    \item[($iii)$] There exists an $n$-vertex graph $G$ such that
  $$\frac{1-\alpha}{2\alpha}I+\frac12J-A_G$$
    is positive semidefinite and has rank at most $d$.
\end{itemize}
In particular, if $\mathscr C$ is a spherical $\{-\alpha,\alpha\}$-code of size $n$, and $G$ is its associated graph, then  $$\frac{1-\alpha}{2\alpha}I+\frac12J-A_G$$
is positive semidefinite and has rank at most $d$.
\end{lemma}
\section{Proof of Theorem \ref{thm:main}}
\subsection{Eigenvalue multiplicity estimates}
In this subsection, we prove the following eigenvalue multiplicity estimates which will be used in the proof of Theorem \ref{thm:main}.
\begin{lemma}\label{lemma:multi}
Let $G=(V,E)$ be a connected graph.
\begin{itemize}
    \item[$(i)$] If $\lambda_2(A_G)=\sqrt{3}$ and $\Delta_G\le 8$, then
\[
m_{\sqrt{3}}(A_G)\le 37.
\]
    \item[$(ii)$] If $\lambda_2(A_G)=\frac{1+\sqrt{5}}{2}$ and $
\Delta_G\le 7$,
then
\[
m_{\frac{1+\sqrt{5}}{2}}(A_G)\le 26.
\]
\end{itemize}
\end{lemma}

We remark that the spectral radius of the star graph $D_4$ equals $\sqrt{3}$ and of the path graph $P_4$ equals $\frac{1+\sqrt{5}}{2}$.
To prove the above lemma, we first recall a separation lemma for induced subgraphs.


\begin{lemma}[{\cite[Lemma 2.2]{balla2024equiangular-exponential-regime}}]\label{lemma:edgedisjoint}
Let $G=(V,E)$ be a connected graph. If $V_1$ and $V_2$ are disjoint with no edge between $V_1$ and $V_2$, then 
\(\lambda_1(A_{G[V_1]})<\lambda_2(A_G)\), or \(\lambda_1(A_{G[V_2]})<\lambda_2(A_G)\), or
\[
\lambda_1(A_{G[V_1]})=\lambda_1(A_{G[V_2]})=\lambda_2(A_G).
\]
\end{lemma}
\begin{proof}[Proof of Lemma \ref{lemma:multi}]
$(i)$ If $|V|\le 37$, the conclusion is immediate. Therefore, we assume that $|V|\ge 38$.

\textbf{Case 1:} $\Delta_G\ge 3$.
 In this case, $G$ contains the graph $D_5$ as a subgraph. Let $V_0$ be its vertex set. By the monotonicity of the spectral radius, and Lemma \ref{thm:smith}, we have
\[
\lambda_1(A_{G[V_0]})\geq\lambda_1\left(A_{D_5}\right)>\lambda_1(A_{D_4})=\sqrt{3}.
\]
Since every vertex of $V_0$ has degree at most $8$, we obtain
\[
|\overline N_G(V_0)|\le 37.
\]
Since there is no edge between $V_0$ and $V\setminus \overline N_G(V_0)$, by Lemma \ref{lemma:edgedisjoint}, we deduce
\[
\lambda_1\bigl(A_{G[V\setminus \overline N_G(V_0)]}\bigr)<\sqrt{3}.
\]
The Cauchy interlacing theorem implies that the multiplicity of the eigenvalue $\sqrt{3}$ in $A_G$ is at most $|\overline N_G(V_0)|$, and hence at most $37$.

\textbf{Case 2:} $\Delta_G\le 2$.
In this case, $G$ is a path graph or a cycle graph.  We derive from Lemma \ref{thm:smith} and Lemma \ref{lemma:Cn} that
\[
m_{\sqrt{3}}(A_G)\le 2<37.
\]
This finishes the proof of $(i)$.

$(ii)$ We argue as in $(i)$ and only indicate the changes.
It suffices to assume that \(\lvert V\rvert\geq 27\). If \(\Delta_G\geq 3\), then \(G\) contains \(D_4\) as a subgraph. Let \(V_0\) be its vertex set. By the degree bound, we have
\(\lvert \overline N_G(V_0)\rvert\leq 26.\)
Moreover,
\(\lambda_1(D_4)=\sqrt{3}>\frac{1+\sqrt{5}}{2},
\)
and hence, by the same argument using Lemma \ref{lemma:edgedisjoint},
\[
  \lambda_1\left(A_{G[V\setminus \overline N_G(V_0)]}\right)
  <\frac{1+\sqrt{5}}{2}.
\]
The Cauchy interlacing theorem therefore gives
\(m_{\frac{1+\sqrt{5}}{2}}(A_G)\leq 26.
\)
If \(\Delta_G\leq 2\), then \(G\) is a path or a cycle, and
\(m_{\frac{1+\sqrt{5}}{2}}(A_G)\leq 2<26.
\)
This proves $(ii)$.
\end{proof}







\subsection{Bounding the maximum degree}
In Lemma \ref{lemma:multi}, we need the condition on $\Delta_G$. In this subsection, we shows that when the cardinality of the equiangular lines is sufficiently large, we can find an associated graph which satisfies the maximum degree condition. In fact, we prove the following theorem.
\begin{theorem}\label{thm:maximum-degree}
Let $\mathcal{L}=\{l_1,l_2,\ldots,l_n\}$ be  $n$ equiangular lines in $\mathbb{R}^d$ with common angle $\arccos(\alpha)$. 
\begin{itemize}
    \item[(i)] If $\alpha=\frac{1}{1+2\sqrt{3}}$ and $n>169$, then there exist a spherical $\{-\alpha,\alpha\}$-code of size $n$ in $\mathbb{R}^d$ such that the maximum degree of its associated graph does not exceed $8$.
\item[(ii)] If $\alpha=\sqrt{5}-2$ and $n>135$, then there exist a spherical $\{-\alpha,\alpha\}$-code of size $n$ in $\mathbb{R}^d$ such that the maximum degree of its associated graph does not exceed $7$.
\end{itemize}
\end{theorem}
To prove this theorem, we first show the following lemma.

\begin{lemma}\label{lemma:maximum-degree}
Let $\mathscr C$ be a spherical $\left\{-\alpha,\alpha\right\}$-code of size $n$ and $G$ be its associated graph. Assume that the eigenspace of $\lambda_1(M_{\mathscr C})$ contains a nonzero function with all coordinates nonnegative.
\begin{itemize}
    \item[(i)] If $\alpha=\frac{1}{1+2\sqrt{3}}$, $\lambda_1(M_{\mathscr C})>\frac{1-\alpha^2}{2\alpha^2}=6+2\sqrt{3}$, and $n>169$, then
\[
\Delta_G\leq 8.
\]
    \item[(ii)] If $\alpha=\sqrt{5}-2$, $\lambda_1(M_{\mathscr C})>\frac{1-\alpha^2}{2\alpha^2}=4+2\sqrt{5}$, and $n>135$, then
\[
\Delta_G\leq 7.
\]
\end{itemize}
\end{lemma}
The proof of Lemma \ref{lemma:maximum-degree} is based on the following matrix projection inequality due to Balla {\cite{balla2021equiangular}}.

\begin{theorem}[{\cite[Theorem 2.1]{balla2021equiangular}}]\label{thm:projection}
Let $\alpha\in(0,1)$ and $\mathscr{C}$ be a spherical $\{-\alpha,\alpha\}$-code of size $n$ in $\mathbb{R}^d$. If $x$ is a unit eigenfunction of $M_{\mathscr C}$ with eigenvalue $\lambda$, then for every $u\in \mathscr C$ one has
\[
\lambda\left(\frac{\lambda^2}{\alpha^2n+1-\alpha^2}-\frac{1-\alpha^2}{2\alpha^2}\right)x(u)^2\ge \lambda-\frac{1-\alpha^2}{2\alpha^2}.
\]
\end{theorem}

\begin{remark}\label{rem:positive}
As a consequence of Theorem \ref{thm:projection}, we have 
\[
\frac{\lambda^2}{\alpha^2n+1-\alpha^2}-\frac{1-\alpha^2}{2\alpha^2}>0
\]
whenever $\lambda>\frac{1-\alpha^2}{2\alpha^2}$.
\end{remark}

\begin{proof}[Proof of Lemma \ref{lemma:maximum-degree}]
Let $\lambda=\lambda_1(M_{\mathscr C})$, and  $x$ be a unit eigenfunction corresponding to $\lambda$ with nonnegative coordinates. If $\lambda>\frac{1-\alpha^2}{2\alpha^2}$,  by Theorem~\ref{thm:projection} and Remark~\ref{rem:positive}, for every $u\in \mathscr C$, we have
\begin{equation}\label{eq:low}
x(u)\ge \sqrt{\frac{1-\frac{1-\alpha^2}{2\alpha^2\lambda}}{\frac{\lambda^2}{\alpha^2n+1-\alpha^2}-\frac{1-\alpha^2}{2\alpha^2}}}=:Q(\alpha,\lambda,n),
\end{equation}
for every $u\in \mathscr C$. Using \eqref{eq:M_C} and \eqref{eq:low}, for every vertex $u$, we obtain
\begin{equation*}\label{eq:upper}
\begin{aligned}
\left(\frac{\lambda}{\alpha}-\frac{1-\alpha}{\alpha}\right)Q(\alpha,\lambda,n)
&\le \left(\left(\frac{M_{\mathscr C}}{\alpha}-\frac{1-\alpha}{\alpha}I\right)x\right)(u)\\
&=((J-2A_G)x)(u)\\
&=\langle \mathbf 1,x\rangle-2\sum_{v\in N_G(u)}x(v)\\
&\le \sqrt n-2\deg_G(u)Q(\alpha,\lambda,n),
\end{aligned}
\end{equation*}
This implies
\begin{equation}\label{eq:max}
\Delta_G\le \frac{\sqrt n}{2Q(\alpha,\lambda,n)}-\frac{\lambda}{2\alpha}+\frac{1-\alpha}{2\alpha}.
\end{equation}
Let us define the function $f_{\alpha,\lambda}: [0,\infty)\to \mathbb{R}$ by
\[f_{\alpha,\lambda}(t):=\frac{t\lambda^2}{\alpha^2t+1-\alpha^2}-\frac{1-\alpha^2}{2\alpha^2}t.\]
Then we derive that
\begin{equation}\label{eq:max_f}
   \Delta_G\le \frac{1}{2}\sqrt{\frac{f_{\alpha,\lambda}(n)}{1-\frac{1-\alpha^2}{2\alpha^2\lambda}}}-\frac{\lambda}{2\alpha}+\frac{1-\alpha}{2\alpha}.
\end{equation}
Observe that the function $f'(t)=0$ if and only if \begin{equation}\label{eq:threshold_t0}
    t=t_0:=\frac{\sqrt{2}\lambda}{\alpha}-\frac{1-\alpha^2}{\alpha^2},
\end{equation}
and, whenever $t_0>0$, $f$ increase on $[0,t_0]$ and decreases on $[t_0,\infty)$. Since $\lambda>\frac{1-\alpha^2}{2\alpha^2}$, we have $t_0>0$ once $\alpha\leq 1/\sqrt{2}$.

\medskip
\noindent\textit{Proof of $(i)$.} By $\alpha=\frac{1}{1+2\sqrt{3}}$, we have $\frac{1-\alpha}{2\alpha}=\sqrt{3}$. In order to show $\Delta_G\leq 8$ from \eqref{eq:max_f}, it suffices to prove 
\begin{equation}\label{eq:7.25}
\frac{1}{2}\sqrt{\frac{f_{\alpha,\lambda}(n)}{1-\frac{1-\alpha^2}{2\alpha^2\lambda}}}-\frac{\lambda}{2\alpha}\leq 7.25.
\end{equation}
Observe that $t_0>0$. We divide the proof of \eqref{eq:7.25} into two cases.

\textbf{Case 1:} $\lambda<29$.
In this case, we have
\[
t_0=\sqrt{2}(1-2\sqrt{3})\,\lambda-4(3+\sqrt{3})<170\le n.
\]
Therefore,
\[
f(n)\le f(170)=\frac{170(13+4\sqrt{3})\lambda^2}{182+4\sqrt{3}}-170(6+2\sqrt{3}).
\]
To prove \eqref{eq:7.25}, it is therefore enough to prove
\begin{equation}\label{eq:12leq40}
\frac12\sqrt{\frac{\frac{170(13+4\sqrt{3})\lambda^2}{182+4\sqrt{3}}-170(6+2\sqrt{3})}{1-\frac{6+2\sqrt{3}}{\lambda}}}-\frac{(1+2\sqrt{3})\lambda}{2}\le 7.25.
\end{equation}
Since $\lambda>6+2\sqrt{3}$, the inequality \eqref{eq:12leq40} is equivalent to
\begin{equation}\label{eq:12l40}
\frac{(-12-4\sqrt{3})(13+4\sqrt{3})}{182+4\sqrt{3}}\lambda^3+(73-8\sqrt{3})\lambda^2+(66\sqrt{3}-708.25)\lambda+\frac{841(3+\sqrt{3})}{2}\le 0,
\end{equation}
which indeed holds for all $\lambda>6+2\sqrt{3}$.

\textbf{Case 2:} $\lambda\ge 29$.
In this range, 
\[
f(n)\le f(t_0)
=(13+4\sqrt{3})\lambda^2-(28\sqrt{6}+36\sqrt{2})\lambda+96+48\sqrt{3}.
\]
Hence, to prove \ref{eq:7.25}, it is enough to verify
\[
\frac12\sqrt{\frac{(13+4\sqrt{3})\lambda^2-(28\sqrt{6}+36\sqrt{2})\lambda+96+48\sqrt{3}}{1-\frac{6+2\sqrt{3}}{\lambda}}}-\frac{(1+2\sqrt{3})\lambda}{2}\le 7.25.
\]
This is equivalent to
\begin{equation}\label{eq:12g40}
(-73+8\sqrt{3}+28\sqrt{6}+36\sqrt{2})\lambda^2-\left(\frac{1631}{4}+454\sqrt{3}\right)\lambda-\frac{841(3+\sqrt{3})}{2}\ge 0.
\end{equation}
This inequality holds for every $\lambda\ge 29$.
This finishes the proof of $(i)$.

\medskip
\noindent\textit{Proof of $(ii)$.} Since $\alpha=\sqrt{5}-2$, we have $\frac{1-\alpha}{2\alpha}=\frac{1+\sqrt{5}}{2}$. In order to show $\Delta_G\leq 7$ from \eqref{eq:max_f}, it suffices to prove 
\begin{equation}\label{eq:golden}
\frac{1}{2}\sqrt{\frac{f_{\alpha,\lambda}(n)}{1-\frac{1-\alpha^2}{2\alpha^2\lambda}}}-\frac{\lambda}{2\alpha}<8-\frac{1+\sqrt{5}}{2}. 
\end{equation}
Notice that in this case we still have $t_0>0$.
The proof of \eqref{eq:golden} follows analogously as the proof of $(i)$ by dividing into two cases $\lambda<25$ and $\lambda\geq 25$.
\end{proof}
The next two lemmas are proved in \cite[Lemma 3.1 and Lemma 3.2]{balla2021equiangular}. They verify the hypotheses required to apply Lemma \ref{lemma:maximum-degree}. For the reader's convenience, we recall their proofs here.
\begin{lemma}\label{lemma:nonnegative}
Let $\{l_1,l_2,\ldots,l_n\}$ be $n$ equiangular lines in $\mathbb{R}^d$ with common angle $\arccos(\alpha)$. Then there is a spherical $\{-\alpha,\alpha\}$-code $\mathscr C$ of size $n$ in $\mathbb{R}^d$ such that there exists an eigenfunction of $M_{\mathscr C}$ corresponding to $\lambda_1(M_{\mathscr{C}})$ with nonnegative coordinates.
\end{lemma}

\begin{proof}
Choose arbitrary unit vectors $\mathscr C_0=\{u_1,u_2,\ldots,u_n\}$ such that $u_i$ spanning $l_i$, and let $x$ be an eigenfunction of $M_{\mathscr C_0}$ corresponding to $\lambda_1(M_{\mathscr C_0})$. Let
\[
I:=\{i:x(u_i)<0\}.
\]
Define a new spherical $\{-\alpha,\alpha\}$-code $\mathscr{C}:=\{v_1,v_2,\ldots,v_n\}$ by
\[
v_i=
\begin{cases}
-u_i,& i\in I,\\
\phantom{-}u_i,& i\notin I.
\end{cases}
\]
Define $y=(|x_1|,|x_2|,\ldots,|x_n|)^T$. By direct computation, we have $y$ is an eigenfunction of $M_{\mathscr C}$ corresponding to $\lambda_1(M_{\mathscr C})$, and every coordinate of $y$ is nonnegative.
\end{proof}

\begin{lemma}\label{lemma:uniform-upper}
Let $\alpha\in(0,1)$ and $\mathscr{C}$ be a spherical $\{-\alpha,\alpha\}$-code of size $n$. If $\lambda_1(M_{\mathscr C})\le \frac{1-\alpha^2}{2\alpha^2},$
then
\[
n\le \frac{(1-\alpha^2)(1-2\alpha^2)}{2\alpha^4}.
\]
\end{lemma}

\begin{proof}
Since $M_{\mathscr C}$ is positive semidefinite,
\[
\mathrm{tr}(M_{\mathscr C}^2)=\sum_{i=1}^n\lambda^2_i(M_{\mathscr C})\le \lambda_1(M_{\mathscr C})\,\mathrm{tr}(M_{\mathscr C}).
\]
On the other hand,
\[
\mathrm{tr}(M_{\mathscr C})=n,
\qquad
\mathrm{tr}(M_{\mathscr C}^2)=n(\alpha^2n+1-\alpha^2).
\]
Therefore
\[
n(\alpha^2n+1-\alpha^2)\le \frac{1-\alpha^2}{2\alpha^2}\,n,
\]
which is equivalent to the desired estimate.
\end{proof}
To conclude this section, we prove Theorem \ref{thm:maximum-degree}. 
\begin{proof}[Proof of Theorem \ref{thm:maximum-degree}]
   By Lemma \ref{lemma:nonnegative}, we have there exists a spherical $\{-\alpha,\alpha\}$-code $\mathscr{C}$ of size $n$ such that  the eigenspace of $\lambda_1(M_{\mathscr C})$ contains a nonzero function with all coordinates nonnegative. When $n>169$ and $\alpha=\frac{1}{1+2\sqrt{3}}$ (resp., $n>135$ and $\alpha=\sqrt{5}-2$), we have $\lambda_1(M_{\mathscr{C}})>6+2\sqrt{3}$ (resp., $\lambda_1(M_{\mathscr{C}})>4+2\sqrt{5}$) by Lemma \ref{lemma:uniform-upper}. 

   Let $G$ be the graph associated to $\mathscr{C}$. By Lemma \ref{lemma:maximum-degree}, we have $\Delta_G\leq 8$ (resp., $\Delta_G\leq 7$) if $\alpha=\frac{1}{1+2\sqrt{3}}$(resp., $\alpha=\sqrt{5}-2$). 
\end{proof}

\subsection{Proof}
In this subsection, we prove Theorem \ref{thm:main}. We first give the following lemma.
\begin{lemma}\label{lemma:seondeigen}
    Let $G=(V,E)$ be a graph such that $\lambda_1(A_G)>\lambda_2(A_G)$. Denote by $\{G_i=(V_i,E_i)\}_{i=1}^t$ the connected components of $G$. If $\lambda_2(A_G)I-A_G+\frac{1}{2}J$ is positive semidefinite and $\lambda_1(A_{G_1})=\lambda_1(A_G)$, then $m_{\lambda_2(A_G)}(A_G)=m_{\lambda_2(A_G)}(A_{G_1})$.
\end{lemma}
\begin{proof}
Set $\lambda:=\lambda_2(A_G)$. If $t=1$, then this lemma follows directly. Thus, for the remainder of the proof, we assume that $t\geq2$.

It is enough to prove this lemma by proving $\lambda_1(A_{G_i})< \lambda $ for any $i\geq 2$.

Let $f$ (resp., $g$) be an eigenfunction of $A_{G}$ corresponding to $\lambda_1(A_{G_1})$ (resp., $\lambda_1(A_{G_i})$ for some $2\leq i\leq t$). By the Perron--Frobenius theorem, we assume both $f$ and $g$ have no negative coordinates, and $f$ is positive on $V_1$ and $g$ is positive on $V_i$. Take a constant $c$ such that the function $h:=f-cg$ satisfies $h^T\mathbf{1}=0$. Notice that such a $c$ is non-zero. Since $\lambda I-A_G+\frac{1}{2}J$ is positive semidefinite, we have $h^T(\lambda I-A_G+\frac{1}{2}J)h\geq0$, which is equivalent to $$(\lambda -\lambda_1(A_G))f^Tf\geq c^2(\lambda_1(A_{G_i})-\lambda )g^Tg.$$
 This implies $\lambda_1(A_{G_i})< \lambda$ for any $i\geq 2$ and proves the lemma.
\end{proof}
We next prove Theorem \ref{thm:main}.
\begin{proof}[Proof of Theorem \ref{thm:main}]
$(i)$ We first assume $d\geq 121$.
    
Let $\mathcal{L}=\{l_1,l_2,\ldots,l_n\}$ be $n$ equiangular lines in $\mathbb{R}^d$ with $n=N_{\frac{1}{1+2\sqrt{3}}}(d)$. If $n\leq169$, then Theorem \ref{thm:main} $(i)$ follows directly. So we assume $n>169$. By Theorem \ref{thm:maximum-degree} $(i)$, there exists a spherical $\{-\alpha,\alpha\}$-code of size $n$ in $\mathbb{R}^d$ such that its associated graph $G=(V,E)$ satisfying $\Delta_G\leq8$. By Lemma \ref{lemma:equivalent}, we obtain  $\sqrt{3}I+\frac12J-A_G$ is positive semidefinite and has rank at most $d$. This implies that if $\sqrt{3}$ is an eigenvalue of $A_G$, then it must be either the largest or the second largest eigenvalue of $A_G$. 

If $\sqrt{3}$ is not an eigenvalue of $A_G$, then we derive  
    $$d\geq \mathrm{rank}\left(\sqrt{3} I-A+\frac{1}{2}J\right)\geq \mathrm{rank}(\sqrt{3} I-A)-1\geq  n-1.$$
Since $d\geq 121$, we have $n\leq d+1<\left\lfloor \frac{4(d-1)}{3}\right\rfloor$. 

If $\sqrt{3}$ is the largest eigenvalue of $A_G$. We denote by $\{G_i=(V_i,E_i)\}_{i=1}^t$ the connected components of $G$. We assume $\lambda_1(A_{G_i})=\sqrt{3}$ for $1\leq i\leq l$ and  $\lambda_1(A_{G_i})<\sqrt{3}$ for $l+1\leq i\leq t$. Since no graph of order less than $4$ has spectral radius $\sqrt{3}$, we have $n\geq \sum_{i=1}^l|V_i|\geq 4l,$
and hence
\begin{equation}\label{eq:ngeq3l}
    \mathrm{dim\,ker}(\sqrt{3} I-A)=l\leq \frac{n}{4}.
\end{equation}
 Because both $\sqrt{3}I-A_G$ and $\frac{1}{2}J$ are positive semidefinite, it follows that $$\mathrm{ker}\left(\sqrt{3}I-A+\frac{1}{2}J\right)=\mathrm{ker}(\sqrt{3}I-A)\cap\mathrm{ker}(J).$$
 By the Perron--Frobenius theorem, there exists a nonzero vector $x$ in $\mathrm{ker}(\sqrt{3}I-A)$ without negative coordinates. Because $x\notin \mathrm{ker}(J)$, we have  $$\mathrm{dim\,ker}\left(\sqrt{3}I-A+\frac{1}{2}J\right)\leq \mathrm{dim\,ker}(\sqrt{3} I-A)-1.$$
 By the rank-nullity
theorem, we deduce 
$$\mathrm{rank}(\sqrt{3} I-A)\leq\mathrm{rank}\left(\sqrt{3} I-A+\frac{1}{2}J\right)-1\leq d-1.$$
 Applying the inequality (\ref{eq:ngeq3l}), this leads to  $$n=\mathrm{rank}(\sqrt{3}I-A)+\mathrm{dim\,ker}((\sqrt{3} I-A)\leq d-1+\frac{n}{4}.$$
 That is, we have $n\leq \frac{4(d-1)}{3}$. 

If $\sqrt{3}$ is the second largest eigenvalue of $A_G$, we derive \begin{equation*}
\begin{aligned}
    n&=\mathrm{rank}(\sqrt{3}I-A_G)+\mathrm{dim\,ker}(\sqrt{3}I-A_G)\\
    &\leq\mathrm{rank}\left(\sqrt{3}I-A_G+\frac{J}{2}\right)+1+\mathrm{dim\,ker}(\sqrt{3}I-A_{G})\\
     &\leq d+1+m_{\sqrt{3}}(A_G).
\end{aligned}
\end{equation*}
Denote by $\{G_i=(V_i,E_i)\}_{i=1}^t$ the connected components of $G$. There exists at least one connected component, say $G_1$, with $\lambda_1(A_{G_1})=\lambda_1(A_G)>\sqrt{3}$. By Lemma \ref{lemma:seondeigen}, we have $m_{\sqrt{3}}(A_G)=m_{\sqrt{3}}(A_{G_1})$. Then, we compute
\begin{equation*}
    n\leq d+1+m_{\sqrt{3}}(A_{G})= d+1+m_{\sqrt{3}}(A_{G_1})\leq d+1+37\leq \left\lfloor\frac{4d-4}{3}\right\rfloor,
\end{equation*}
where we have applied Lemma \ref{lemma:multi} $(i)$. This completes the proof of $(i)$ under the condition $d\geq 121$. 

When $d< 121$ ,we have \[N_{\frac{1}{1+2\sqrt{3}}}(d)\leq N_{\frac{1}{1+2\sqrt{3}}}(121)\leq \max\left\{169, \left\lfloor\frac{4d-4}{3}\right\rfloor\right\}=169.\]
This completes the proof of Theorem \ref{thm:main} $(i)$.

 $(ii)$ We first assume $d\geq 88$.
 
 If $n\leq135$, Theorem \ref{thm:main} $(ii)$ follows directly. We assume $n>135$. By Theorem \ref{thm:maximum-degree} and Lemma \ref{lemma:equivalent}, we deduce that there exists a graph $G=(V,E)$ with $\Delta_G\leq 7$, such that $\frac{\sqrt{5}+1}{2}I+\frac12J-A_G$ is positive semidefinite and has rank at most $d$. Arguing as in the proof of Theorem~\ref{thm:main} \((i)\), we obtain $N_{\sqrt{5}-2}(d)\leq \left\lfloor  \frac{4d-4}{3}\right \rfloor$ if $\sqrt{5}-2$ is not the second largest eigenvalue of $A_G$.
 
Assume $\sqrt{5}-2$ is the second largest eigenvalue of $A_G$. Denote by $\{G_i=(V_i,E_i)\}_{i=1}^t$ the connected components of $G$. Then there exists at least one connected component, say $G_1$, with $\lambda_1(A_{G_1})=\lambda_1(A_G)>\frac{1+\sqrt{5}}{2}$. By Lemma \ref{lemma:seondeigen}, we obtain $$N_{\sqrt{5}-2}(d)\leq d+1+m_{\frac{1+\sqrt{5}}{2}}(A_{G_1}).$$
Using Lemma \ref{lemma:multi} $(ii)$, we obtain $$N_{\sqrt{5}-2}(d)\leq d+27\leq \left\lfloor  \frac{4d-4}{3}\right \rfloor.$$

 When $d< 88$, we have $N_{\sqrt{5}-2}(d)\leq N_{\sqrt{5}-2}(88)\leq \max\{135, \left\lfloor\frac{4d-4}{3}\right\rfloor\}=135$.
 This conclude the proof of Theorem \ref{thm:main} $(ii)$.
\end{proof}

\section{Proof of Theorem \ref{thm:counterexample}}
To prove Theorem \ref{thm:counterexample}, we need the following preparation. 
\begin{lemma}\label{lemma:equiangular}
    Let $G=(V,E)$ be a connected $n$-vertex $k$-regular graph with the second largest eigenvalue $\lambda$. Suppose that
    \[
        \lambda-k+\frac{n}{2}\geq 0.
    \]
   Denote by $m:=m_{\lambda}(A_G)$.  Then there exists  $n$ equiangular lines in
    $\mathbb{R}^{n-m}$ with common angle
    \[
        \arccos\left(\frac{1}{1+2\lambda}\right).
    \]
\end{lemma}

\begin{proof}
    Define
    \[
        M:=\lambda I+\frac{1}{2}J-A_G.
    \]
Since \(G\) is \(k\)-regular, the all-ones vector \(\boldsymbol{1}\) is an
eigenvector of \(A_G\) with eigenvalue \(k\). Moreover, since \(G\) is
connected, this eigenvalue is simple, and every eigenvector corresponding to
an eigenvalue \(\lambda_i(A_G)\), \(i\geq 2\), is orthogonal to
\(\boldsymbol{1}\). Hence, if \(f\) is such an eigenvector, then \(Jf=0\).
Therefore
\[
    Mf=\left(\lambda-\lambda_i(A_G)\right)f.
\]
It follows that \(\lambda-\lambda_i(A_G)\) is an eigenvalue of \(M\) for each
\(i\geq 2\), with the same multiplicity as \(\lambda_i(A_G)\) as an eigenvalue
of \(A_G\).

    On the other hand, using the $k$-regularity of $G$, we have
    \[
        M\boldsymbol{1}
        =\left(\lambda I+\frac{1}{2}J-A_G\right)\boldsymbol{1}
        =\left(\lambda+\frac{n}{2}-k\right)\boldsymbol{1}.
    \]
    Therefore, the spectrum of $M$ is
    \[
        \left\{\lambda-k+\frac{n}{2}\right\}
        \cup
        \left\{\lambda-\lambda_i(A_G): 2\leq i\leq n\right\},
    \]
    where eigenvalues are counted with multiplicity.

    Since $\lambda$ is the second largest eigenvalue of $A_G$, we have
    \[
        \lambda-\lambda_i(A_G)\geq 0
        \qquad\text{for all } i\geq 2.
    \]
    Together with the assumption
    \[
        \lambda-k+\frac{n}{2}\geq 0,
    \]
    this implies that $M$ is positive semidefinite. Moreover, the nullity of $M$ is at least
    $m$, and hence
    \[
        \operatorname{rank}(M)\leq n-m.
    \]
    By Lemma~\ref{lemma:equivalent}, there exist $n$ equiangular lines in
    $\mathbb{R}^{n-m}$ with common angle
    \[
        \arccos\left(\frac{1}{1+2\lambda}\right).
    \]
\end{proof}
We now turn to the proof of Theorem \ref{thm:counterexample}.
\begin{proof}[Proof of Theorem \ref{thm:counterexample}]
Since $\alpha_n=\frac{1}{1+4\cos\frac{\pi}{n}}$, we have $\lambda_n=\frac{1-\alpha_n}{2\alpha_n}=2\cos\frac{\pi}{n}$. For any odd integer $n\geq 3$, by Lemma \ref{thm:smith}, the only graph with spectral radius $\lambda_n$ is $P_{n-1}$. This implies $\kappa_{\lambda_n}=n-1$. By Lemma \ref{lemma:Cn}, $\lambda_n$ is the second largest eigenvalue of $A(C_{2n})$ with multiplicity $2$. Because $C_{2n}$ is a connected $2n$-vertex $2$-regular graph with $\lambda_n-2+n\geq 0$, there exist $2n$ equiangular lines in $\mathbb{R}^{2n-2}$ with common angle $\arccos(\alpha_n)$ by Lemma \ref{lemma:equiangular}. By a direct computation, we have $2n>\frac{\kappa_{\lambda_n}(2n-3)}{\kappa_{\lambda_n}-1}$, once $n\geq 4$. Indeed, \[2n-\left\lfloor\frac{(n-1)(2n-3)}{n-2}\right\rfloor=1,\ \text{for any $n\geq 4$}.\] 

On the other hand, we observe for any $n\geq 2$ that $\alpha_n>\frac{1}{5}$, and hence $\frac{(1-\alpha_n^2)(1-2\alpha_n^2)}{2\alpha_n^4}<276$. Set $n_0=138$. Then for any odd integer $n>n_0$, we derive $$N_{\alpha_n}(2n-2)\geq2n> \max\left\{\frac{(1-\alpha_n^2)(1-2\alpha_n^2)}{2\alpha_n^4},\left\lfloor\frac{\kappa_{\lambda_n}(2n-3)}{\kappa_{\lambda_n}-1}\right\rfloor\right\}.$$
\end{proof}
\section{Concluding remarks}
In this paper, we prove that Balla's Conjecture \ref{conj:balla} holds for $\alpha\in\left\{\frac{1}{1+2\sqrt{3}},\sqrt{5}-2\right\}$ and fails for  infinitely many $\alpha$. As Balla noted, "Even if this ambitious conjecture is false in general, it would already be very interesting to
see if it holds in the case where $1/\alpha$ is an odd integer." Thus, the next natural and important question is whether Conjecture \ref{conj:balla} holds when $1/\alpha$ is odd. This leads to the following conjecture.
\begin{Conj}[Balla's conjecture for $1/\alpha$ is odd]
  For integer $k\geq 1$, we have $$N_{\frac{1}{2k+1}}(d)\leq \max\left\{\binom{4k^2+4k}{2},\left\lfloor\frac{(k+1)(d-1)}{k}\right\rfloor\right\},$$
for any $d\geq 1$.
\end{Conj}
For general $\alpha$, it is natural to ask the following question.
\begin{Question}
    Characterize all real numbers $\alpha\in(0,1)$ for which Conjecture \ref{conj:balla} holds.
\end{Question}
Notably, our counterexamples exceed the upper bound in Conjecture \ref{conj:balla} by only one. This leads us to pose the following question.
\begin{Question}
  Dose there exist an absolute constant $C$, such that for any $\alpha\in (0,1)$ with $\kappa_{\frac{1-\alpha}{2\alpha}}<\infty$, we have $$N_{\alpha}(d)\leq \max\left\{\frac{(1-\alpha^2)(1-2\alpha^2)}{2\alpha^4},\left\lfloor\frac{\kappa_{\frac{1-\alpha}{2\alpha}}(d-1)}{\kappa_{\frac{1-\alpha}{2\alpha}}-1}\right\rfloor+C\right\},$$
for any $d\geq 1$?
\end{Question}

\section*{Acknowledgement}
C.G. is supported by the Start-up Research Fund of Fuzhou University for the project "Research on Related Problems of Equiangular Line Sets Problem" (No. 511813-XRC-26048).
S.L. is supported by the Scientific Research Innovation Capability Support Project for Young Faculty SRICSPYF-ZY2025160 and the National Natural Science Foundation of China No. 12431004. 

\bibliographystyle{abbrv}
\bibliography{references}
\end{document}